\documentclass[review]{elsarticle}

\usepackage{lineno,hyperref}
\nolinenumbers

\journal{Journal of \LaTeX\ Templates}
\usepackage{epsfig,amssymb,latexsym,color,amsmath,pifont,colordvi,multicol}
\usepackage{subfigure}
\usepackage{amssymb}
\usepackage{dsfont}
\usepackage{setspace}
\usepackage{bbm}
\usepackage{wrapfig}
\usepackage{algorithm}
\usepackage{algpseudocode}

\journal{Journal of Nonlinear Science}
\definecolor{backgrey}{rgb}{0.86,0.86,0.86}
\definecolor{dblue}{rgb}{0,0.0,0.5}
\definecolor{dred}{rgb}{0.4,0.2,0}
\definecolor{dgreen}{rgb}{0.0,0.5,0}

\renewcommand{\baselinestretch}{1}

\newcommand{\captionfonts}{\small}
\makeatletter  
\long\def\@makecaption#1#2{%
  \vskip\abovecaptionskip
  \sbox\@tempboxa{{\captionfonts #1: #2}}%
  \ifdim \wd\@tempboxa >\hsize
    {\captionfonts #1: #2\par}
  \else
    \hbox to\hsize{\hfil\box\@tempboxa\hfil}%
  \fi
  \vskip\belowcaptionskip}
\makeatother   
\newtheorem{theorem}{Theorem}

\newtheorem{remark}[theorem]{Remark}

\newtheorem{definition}[theorem]{Definition}

\newtheorem{proposition}[theorem]{Proposition}

\newenvironment{proof}[1][Proof]{\textbf{#1.} }{\ \hspace*{\fill} \rule{0.5em}{0.5em}}

\newcommand{\K}{\bf K}
\newcommand{\G}{\bf G}
\newcommand{\A}{\bf A}

\begin{document}

\begin{frontmatter}

\title{Online Real-time Learning of Dynamical Systems from Noisy Streaming Data: A Koopman Operator Approach}

\author{S. Sinha*}\cortext[mycorrespondingauthor]{Corresponding author}
\ead{subhrajit.sinha@pnnl.gov}
%
%
\author{S.~P. Nandanoori}
\ead{saipushpak.n@pnnl.gov}
\author{D.~A. Barajas-Solano}
\ead{David.Barajas-Solano@pnnl.gov}

\fntext[fn1]{The authors are with the Pacific Northwest National Laboratory,
Richland, WA, 99354}

\begin{abstract}
Recent advancements in sensing and communication facilitate obtaining high-frequency real-time data from various physical systems like power networks, climate systems, biological networks, etc. However, since the data are recorded by physical sensors, it is natural that the obtained data is corrupted by measurement noise. In this paper, we present a novel algorithm for online real-time learning of dynamical systems from noisy time-series data, which employs the Robust Koopman operator framework to mitigate the effect of measurement noise.
The proposed algorithm has three main advantages: a) it allows for online real-time monitoring of a dynamical system; b) it obtains a linear representation of the underlying dynamical system, thus enabling the user to use linear systems theory for analysis and control of the system; c) it is computationally fast and less intensive than the popular Extended Dynamic Mode
Decomposition (EDMD) algorithm.
We illustrate the efficiency of the proposed algorithm by applying it to identify the Van der Pol oscillator, the IEEE 68 bus system, and a ring network of Van der Pol oscillators.

\end{abstract}


\end{frontmatter}

\nolinenumbers

\section{Introduction}

The field of dynamical systems was born with the works of Sir Issac Newton \cite{principia_newton}. Since its humble beginnings, it has developed into a matured branch of pure mathematics with applications to most branches of science and engineering.
For almost three hundred years, the approach to studying dynamical systems was model-based, in which the system under study is modeled from first principles.
However, this approach has severe limitations when studying systems like power networks, biological networks, stock markets, etc., which are inherently highly nonlinear and typically large-scale. As such it is almost impossible to model these dynamical systems from first principles. The way to circumvent this problem is to use data-driven techniques, in which the governing equations of motion are learned directly from time-series data of the states of the system.

Of the various data-driven learning methods, in recent years, the Koopman operator approach~\cite{Lasota,mezic_koopmanism,EDMD_williams,mezic_koopman_stability,sinha_robust_dmd_journal,sinha_equivariant_ifac,nandanoori2020data,nandanoori2021data} has been used extensively for the data-driven learning of dynamical systems.
Compared to machine learning (ML) methods, the Koopman operator framework has multiple advantages, with the major factor being the fact that the amount of training data required by Koopman operator approaches is far less compared to ML methods.
Furthermore, the Koopman operator is a linear operator that generates an equivalent linear representation of the nonlinear system, which allows us to use concepts from linear systems theory for the analysis and control design for nonlinear systems~\cite{sinha_data_driven_control_arxiv,huang2018feedback}. 

Given a dynamical system, instead of looking at the evolution of trajectories, the Koopman operator governs the evolution of functions on the state space under the dynamical map. The main advantage of this framework, as mentioned earlier, is the fact that the Koopman operator is a linear operator, but the trade-off is that it is an infinite-dimensional operator.
Hence, given any general nonlinear system, the corresponding Koopman operator generates an infinite-dimensional linear system, which is an exact representation of the nonlinear system and, unlike linearization, is valid globally on the state space. Furthermore, the Koopman framework facilitates data-driven learning of dynamical systems \cite{sinha_sparse_koopman_acc,nandanoori2020data} and has been extensively used for data-driven learning in synthetic biology \cite{sootla2017pulse,harrison2021stability}, building systems \cite{eisenhower2010decomposing,korda_mezic_predictor}, climate systems \cite{slawinska2019quantum}, robotics \cite{bruder2019modeling,abraham2019active}, neuroscience \cite{marrouch2020data}, and power systems \cite{susuki2016applied,sinha2019information,nandanoori2022graph}, where a finite-dimensional approximation of the Koopman operator is computed from time-series data of the dynamical states.

However, when it comes to using data-driven learning of real physical systems, it is often a requirement that the learning framework learns the dynamics in an online fashion and in real-time.
This is especially relevant for power networks, stock markets, building systems etc., which demand constant monitoring to ensure efficient and reliable operation.
For example, a power network is exposed to multiple adversarial events like sudden load changes, faults and attacks~\cite{kropp2006system} and as such, it is imperative to monitor the power network constantly to identify any adversarial event as soon as possible and to take appropriate measures to ensure safe operation.
Therefore, there is a need for a framework that can learn the underlying dynamics online and in real time~\cite{sinha_online_learning_PES}.
Moreover, since the data is obtained from physical sensors, it is often the case that the data obtained is corrupted by noise.

In this paper, we aim to address the problem of data-driven learning from noisy data in an online fashion and in real-time.
In particular, we use the Robust Koopman operator framework~\cite{robust_DMD_ACC,sinha_robust_dmd_journal} to mitigate the effect of measurement noise.
Furthermore, we propose a recursive algorithm to learn the Koopman representation of the underlying system, so that the Koopman operator gets updated as each new data point streams in.

This manuscript is organized as follows: In section~\ref{sec_transfer_operator} we discuss the basics of transfer operators for dynamical systems, followed by the formulation of the Robust Koopman operator computation algorithm in section~\ref{sec_robust_koopman}.
The main results of the paper are presented in section~\ref{sec_online_learning}, where we derive and describe the online learning algorithm.
In section~\ref{sec_simulation} we illustrate the efficiency of the proposed algorithm by applying it to three different dynamical systems, namely, a Van der Pol oscillator, the IEEE 68 bus system and a ring network of Van der Pol oscillators. Finally, we conclude the paper in section~\ref{sec_conclusions}.

\section{Transfer Operators}\label{sec_transfer_operator}

Consider a discrete-time dynamical system
\begin{eqnarray}\label{system}
x_{t+1} = T(x_t)
\end{eqnarray}
where $T:X\subset \mathbb{R}^N \to Z$ is assumed to be at least of class ${\cal C}^1$.  Associated with the dynamical system (\ref{system}) is the Borel-$\sigma$ algebra ${\cal B}(X)$ on $X$ and the vector space ${\cal M}(X)$ of bounded complex-valued measures on $X$. With this, two linear operators, namely, the Perron-Frobenius (P-F) and the Koopman operator, can be defined as follows~\cite{Lasota}:
\begin{definition}[Perron-Frobenius Operator] 
The Perron-Frobenius operator $\mathbb{P}:{\cal M}(X)\to {\cal M}(X)$ is given by
\[[\mathbb{P}\mu](A)=\int_{{X} }\delta_{T(x)}(A)d\mu(x)=\mu(T^{-1}(A)),\]
where $\delta_{T(x)}(A)$ is the stochastic transition function that measures the probability that point $x$ will reach the set $A$ in one-time step under the system mapping $T$. 
\end{definition}

\begin{definition}[Invariant measures] Invariant measures are the fixed points of
the P-F operator $\mathbb{P}$ that are also probability measures. Let $\bar \mu$ be the invariant measure then, $\bar \mu$ satisfies
\[\mathbb{P}\bar \mu=\bar \mu.\]
\end{definition}

If the state space $Z$ is compact, it is known that the P-F operator admits at least one invariant measure.

\begin{definition} [Koopman Operator] 
Given any $h\in\cal{F}$, $\mathbb{U}:{\cal F}\to {\cal F}$ is defined by
\[[\mathbb{U} h](x)=h(T(x))\]
where $\cal F$ is the space of function (observables) invariant under the action of the Koopman operator.
\end{definition}

\begin{figure}[htp!]
\centering
\includegraphics[scale=.45]{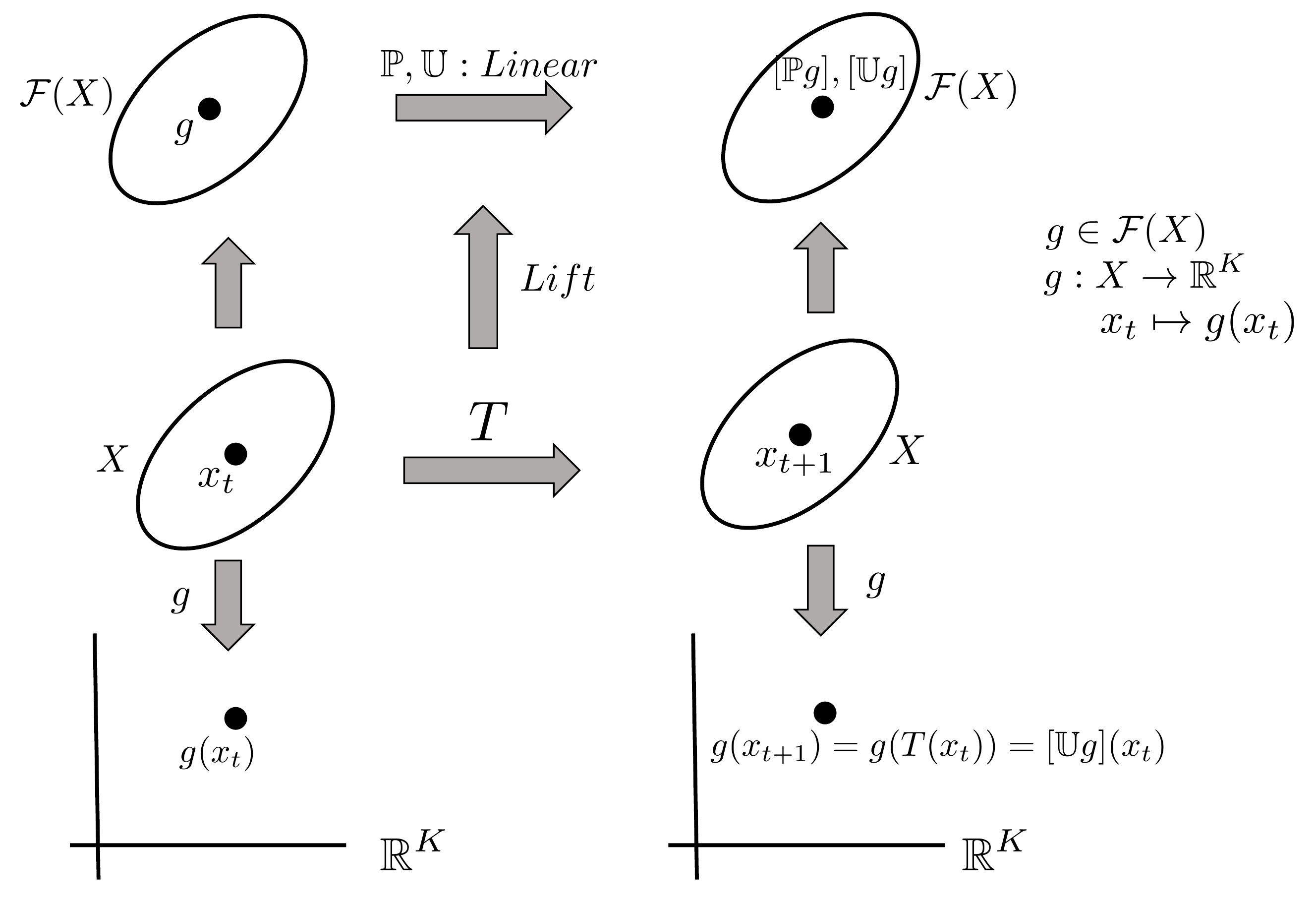}
\caption{Perron-Frobenius and Koopman operators corresponding to a dynamical system.}\label{koopman_diagram}
\end{figure}
Both the P-F and the Koopman operators are linear operators even if the underlying system is nonlinear and while the analysis is made tractable by linearity, the tradeoff is that these operators are typically infinite-dimensional.
In particular, the P-F  and Koopman operators lift a dynamical system from a finite-dimensional space to generate an infinite-dimensional linear system on the space of distributions and functions respectively.

\section{Robust Koopman Operator}\label{sec_robust_koopman}

The Koopman operator is infinite-dimensional, but for computational purposes, one has to compute a finite-dimensional approximation of the Koopman operator from time-series data. The most commonly used approach for computing the finite-dimensional approximation is known as Extended Dynamic Mode Decomposition (EDMD) \cite{EDMD_williams}.

\subsection{Extended Dynamic Mode Decomposition}
Let 
\[[x_1,x_2, \cdots , x_M]\in\mathbb{R}^{N\times M}\] 
be $M$ data points from a $N$-dimensional system $x_t\mapsto T(x_t)$, which evolves on $\mathbb{R}^N$.
Let 
\[{\bf \Psi}(x)=[\psi_1(x), \psi_2(x), \cdots , \psi_K(x)]:\mathbb{R}^N\to \mathbb{R}^K\]
be the set of vector-valued Koopman observables such that each $\psi_i:\mathbb{R}^N\to \mathbb{R}$ is square-integrable.
Let 
\begin{eqnarray}\label{G_and_A_formula}
{\bf G} = \frac{1}{M}\sum_{i=1}^{M-1}{\bf \Psi}(x_i)^\top{\bf \Psi}(x_i);
\quad {\bf A} = \frac{1}{M}\sum_{i=1}^{M-1}{\bf \Psi}(x_i)^\top{\bf \Psi}(x_{i+1}).
\end{eqnarray}

Then the finite-dimensional approximation ${\bf K}\in\mathbb{R}^{K\times K}$ is obtained as the solution of the optimization problem
\begin{eqnarray}\label{EDMD_opt}
\min_{\bf K}\parallel {\bf G}{\bf K}-{\bf A}\parallel,
\end{eqnarray}
where $\bf G$ and $\bf A$ are given by (\ref{G_and_A_formula}).

\subsection{Robust Koopman Operator}
The standard approach to compute the Koopman operator from time-series data is by solving the optimization problem (\ref{EDMD_opt}). However, in presence of noise in the data, (\ref{EDMD_opt}) leads to erroneous Koopman operators and one has to explicitly take into account the noise in the data, while computing the Koopman operator. We assume that the measurements are of the form
\begin{eqnarray}
x_t^m = x_t + \delta x_t
\end{eqnarray}
where $x_t^m$ is the measured value of the states at time $t$, $x_t$ is the true value of the states at time $t$ and $\delta x_t$ is the noise in the measurement at time $t$. We further assume $\parallel\delta x_i\parallel \leq \Lambda$ for some $\Lambda>0$, that is $\delta x_i\in\Delta$, where $\Delta$ is a compact subset of $\mathbf{R}^N$. The uncertainty in the data acts as an adversary which tries to maximize the residual error of the optimization problem which computes the Koopman operator, and hence to obtain the Koopman operator $\bf K$ for uncertain data, a robust optimization problem is formulated as the following $\min-\max$ optimization problem \cite{robust_DMD_ACC,sinha_robust_dmd_journal}. 

\begin{equation}\label{edmd_robust}
\min\limits_{\bf K}\max_{\delta x_1,\cdots , \delta x_M\in \Delta}\parallel {\bf G}_\delta{\bf K}-{\bf A}_\delta\parallel_F=:\min\limits_{\bf K}\max_{\delta\in \Delta} {\cal J}({\bf K}, {\bf G}_\delta,{\bf A}_\delta)
\end{equation}
where
\begin{eqnarray}
&&{\bf G}_\delta=\frac{1}{M}\sum_{i=1}^{M-1} \boldsymbol{\Psi}({x}_i)^\top \boldsymbol{\Psi}({x}_i+\delta x_i)\nonumber\\
&&{\bf A}_\delta=\frac{1}{M}\sum_{i=1}^{M-1} \boldsymbol{\Psi}({x}_i)^\top \boldsymbol{\Psi}({x}_{i+1}+\delta x_{i+1}),
\end{eqnarray}
with ${\bf K},{\bf G}_\delta,{\bf A}_\delta\in\mathbb{C}^{K\times K}$ and $\Delta$ is the bounded uncertainty set.

The robust optimization problem (\ref{edmd_robust}), is in general non-convex because the cost $\cal J$ may not be a convex function of $\delta$. 

\begin{proposition}
The optimization problem (\ref{edmd_robust}) can be approximated as
\begin{equation}\label{edmd_robust_convex}
\min\limits_{\bf K}\max_{\delta{\bf G},\delta{\bf A}\in  {\cal U}}\parallel ({\bf G}+\delta {\bf G}){\bf K}-({\bf A}+\delta {\bf A})\parallel_F
\end{equation}
where $\cal U$ is a compact set in $\mathbb{R}^{K\times K}$.
\end{proposition}
\begin{proof}
From Taylor series expansion we have, ${\bf \Psi}(x_i+\delta x_i) = {\bf \Psi}(x_i) +  {\bf \Psi}'(x_i) \delta x_i+h.o.t.$, where ${\bf \Psi}'(x_i)$ is the first derivative of ${\bf \Psi}(x)$ at $x_i$. Hence,
\begin{eqnarray*}
{\bf G}_{\delta} &\approx& {\bf G} + \frac{1}{M}\sum_{i=1}^{M-1}{\bf \Psi}^\top (z_i)\delta z_i {\bf \Psi}'(z_i)\\
&=& {\bf G} +\delta {\bf G}
\end{eqnarray*}
where $\delta {\bf G} = \frac{1}{M}\sum_{i=1}^{M}{\bf \Psi}^\top (x_i)\delta z_i {\bf \Psi}'(x_i)$.

Moreover,
\begin{eqnarray*}
&&\parallel \delta {\bf G} \parallel_F = \parallel \frac{1}{M}\sum_{i=1}^{M}{\bf \Psi}^\top (x_i)\delta x_i {\bf \Psi}'(x_i) \parallel_F\\
&\leq& \frac{1}{M}\sum_{i = 1}^{M}\parallel {\bf \Psi}^\top (x_i)\delta x_i {\bf \Psi}'(x_i)\parallel_F \\
&\leq&\frac{1}{M}\sum_{i=1}^{M}\parallel {\bf \Psi}^\top (x_i)\parallel_F \cdot \parallel\delta x_i \parallel_F \cdot \parallel {\bf \Psi}'(z_i)\parallel_F
\end{eqnarray*}
Hence, $\delta {\bf G}$ belongs to a compact set ${\cal U}_1$. Similarly, one can show ${\bf A}_\delta \approx {\bf A} + \delta {\bf A}$ and $\delta {\bf A}$ belongs to a compact set ${\cal U}_2$. Letting ${\cal U}={\cal U}_1\cup{\cal U}_2$, proves the proposition.
\end{proof}

With this, we have the following theorem, which derives a regularized least-squares optimization problem whose solution is the Robust Koopman Operator.

\begin{theorem}
The optimization problem 
\begin{equation}
\min\limits_{\bf K}\max_{\delta{\bf G},\delta{\bf A}\in  {\cal U}}\parallel ({\bf G}+\delta {\bf G}){\bf K}-({\bf A}+\delta {\bf A})\parallel_F
\end{equation}
is equivalent to the following optimization problem
\begin{eqnarray}\label{rob_eqv}
\min\limits_{\bf K}\parallel {\bf G}{\bf K}-{\bf A}\parallel_F+\lambda \parallel {\bf K}\parallel_F,
\end{eqnarray}
where $\lambda>0$ is a constant dependent on the bound of the noise.
\end{theorem} 
\begin{proof}
For a $K\times K$ matrix $M=[m_{i,j}]\in\mathbb{R}^{K\times K}$, let ${\cal M}$ denote the vector 
\begin{eqnarray*}
{\cal M}=[m_{1,1},\cdots ,m_{K,1},m_{1,2},\cdots ,m_{K,2},\cdots , m_{K,K}]^\top.
\end{eqnarray*}
This follows from the fact that $\mathbb{R}^{K\times K} \cong \mathbb{R}^{K^2}$. Hence, 
\[\parallel M\parallel_F = \parallel {\cal M}\parallel_2.\]
For two matrices $A$ and $B$, let $A \otimes B$ denote the Kronecker product of $A$ and $B$.
Let ${\cal K}$ be the vector form of $\bf K$ and let $\cal A$ and $\delta {\cal A}$ be the vector forms of $A$ and $\delta A$, respectively.

Then, the min-max optimization problem can be written as
{\small
\begin{eqnarray}\label{opt_2_norm}\nonumber
&&{\cal J} = \min_{\bf K}\max_{\delta{\bf G},\delta{\bf A} \in {\cal U}}\parallel ({\bf G}+\delta {\bf G}){\bf K}-({\bf A}+\delta {\bf A})\parallel_F\\ \nonumber
&=& \min_{\cal K}\max_{\delta{\bf G},\delta{\cal A}\in {\cal U}}\parallel [({\bf G}+\delta {\bf G})\otimes I_K]{\cal K} - ({\cal A}+\delta {\cal A})\parallel_F\\
&=& \min_{\cal K}\max_{\delta{\bf G},\delta{\cal A}\in {\cal U}}\parallel [({\bf G}+\delta {\bf G})\otimes I_K]{\cal K} - ({\cal A}+\delta {\cal A})\parallel_2
\end{eqnarray}
}
where $I_K$ is the $K\times K$ identity matrix.
Writing ${\bf G}\otimes I_K$ as $\hat{G}$ and $\delta {\bf G}\otimes I_K$ as $\delta\hat{G}$, the optimization problem (\ref{opt_2_norm}) can be written as

\begin{eqnarray}\label{opt_2_norm_modified}
{\cal J} = \min_{\cal K}\max_{\substack{\delta{\hat{G}}\in \Pi_K{\cal U}\\ \delta{\cal A}\in {\cal U}}}\parallel ({\hat{G}}+\delta {\hat{G}}){\cal K} - ({\cal A}+\delta {\cal A})\parallel_2,
\end{eqnarray}
where $\Pi_K{\cal U}$ is the projection of $\cal U$ on $\mathbb{R}^{K\times K}\otimes \mathbb{R}^{K\times K}$.

Fix ${\bf K}\in\mathbb{R}^{K\times K}$ and let
\begin{eqnarray}\label{worst_residual_geq}
r = \max_{\substack{\delta{\hat{G}}\in \Pi_K{\cal U}\\ \delta{\cal A}\in  {\cal U}}}\parallel ({\hat{G}}+\delta {\hat{G}}){\cal K} - ({\cal A}+\delta {\cal A})\parallel_2
\end{eqnarray}
be the worst-case residual. Then,
\begin{eqnarray}\label{worst_residual_leq}
\begin{aligned}
r &\leq \max_{\substack{\delta{\hat{G}}\in \Pi_K\bar{\Delta}\\ \delta{\cal A}\in  {\cal U}}}\parallel{\hat G}{\cal K} - {\cal A} \parallel_2 + \parallel\delta {\hat G}{\cal K} - \delta {\cal A} \parallel_2\leq \parallel{\hat G}{\cal K} - {\cal A} \parallel_2 + \lambda \parallel{\cal K} - \mathds{1} \parallel_2\\
&\leq \parallel{\hat G}{\cal K} - {\cal A} \parallel_2 + \lambda \sqrt{\parallel {\cal K}\parallel_2^2 + K}= \parallel {\bf G} {\bf K} - {\bf A} \parallel_F + \lambda \sqrt{\parallel {\bf K}\parallel_F^2 + K}
\end{aligned}
\end{eqnarray}

Again, choose $[\delta \hat{G} \quad \delta \hat{A}]$ as 
\[[\delta \hat{G} \quad \delta \hat{A}]= \frac{\lambda u}{\sqrt{\parallel {\cal K}\parallel_2^2 + K}}[{\cal K}^\top\quad K],\]
where 
\begin{equation}
  u =
    \begin{cases}
      \frac{\hat{G}{\cal K}-\hat{A}}{\parallel \hat{G}{\cal K}-\hat{A} \parallel}, \textnormal{ if } \hat{G}{\cal K}\neq\hat{A}\\
      \textnormal{any unit norm vector otherwise.}
    \end{cases}       
\end{equation}
%
%
Then,
\begin{eqnarray}\label{worst_residual_geq1}
\begin{aligned}
r &= \max_{\substack{\delta{\hat{G}}\in \Pi_K\bar{\Delta}\\ \delta\hat{A}\in  {\cal U}}}\parallel ({\hat{G}}{\cal K}-\hat{A}) + (\delta {\hat{G}}{\cal K}-\delta \hat{A})\parallel_F\\
& = \max_{\substack{\delta{\hat{G}}\in \Pi_K\bar{\Delta}\\ \delta\hat{A}\in  {\cal U}}} \left \| ({\hat{G}}{\cal K}-\hat{A}) + \lambda \left ( \frac{\hat{G}{\cal K}-\hat{A}}{\parallel \hat{G}{\cal K}-\hat{A} \parallel}{\cal K}^\top {\cal K}+K\frac{\hat{G}{\cal K}-\hat{A}}{\parallel \hat{G}{\cal K}-\hat{A} \parallel} \right ) \right \|_F\\
&\geq \parallel ({\hat{G}}{\cal K}-\hat{A})\parallel_F +\lambda \left \| \left ( \frac{\hat{G}{\cal K}-\hat{A}}{\parallel \hat{G}{\cal K}-\hat{A} \parallel}{\cal K}^\top {\cal K} +K\frac{\hat{G}{\cal K}-\hat{A}}{\parallel \hat{G}{\cal K}-\hat{A} \parallel} \right ) \right \|_F \\ 
&\geq \parallel ({\hat{G}}{\cal K}-\hat{A})\parallel_F + \lambda \sqrt{{\cal K}^\top {\cal K}+K}= \parallel {\bf G} {\bf K} - {\bf A} \parallel_F + \lambda \sqrt{\parallel {\bf K}\parallel_F^2 + K}
\end{aligned}
\end{eqnarray}

Hence, from (\ref{worst_residual_leq}) and (\ref{worst_residual_geq1}), the worst case residual is 
\begin{eqnarray}\label{worst_res_eqv}
r = \min_{{\bf K}}\parallel {\bf G}{\bf K} - {\bf A} \parallel_F + \lambda \sqrt{\parallel {\bf K}\parallel_F^2 +K}.
\end{eqnarray}
Since, $K$ is a constant, the $\bf K$ that minimizes $r$ in (\ref{worst_res_eqv}) is the same $\bf K$ that minimizes
\[\parallel {\bf G}{\bf K} - {\bf A} \parallel_F + \lambda \parallel {\bf K}\parallel_F .\]
\end{proof}

\section{Online Learning of Robust Koopman Operator}\label{sec_online_learning}

The Koopman operator is generally computed by solving the optimization problem (\ref{rob_eqv}) or directly using the formula
\begin{equation}\label{closed_form_noisy_koopman}
\K = (\G + \lambda I)^{-1} \A,
\end{equation}
where $I$ is the identity matrix of appropriate dimensions. Note that for computing the Koopman operator one needs to use the entire dataset and hence when a new data point streams in the Koopman operator needs to be recalculated from scratch using the new larger data set.
However, this requires the inversion of a matrix, which is computationally expensive, especially for data-sets obtained from large-dimensional systems. This necessitates developing a recursive algorithm for Robust Koopman operator computation.
In this section, we describe an algorithm that computes the Koopman operator recursively and thus reducing the computational cost.

Let
\begin{eqnarray}
X_M = [x_1,x_2,\ldots,x_M],& Y_M = [y_1,y_2,\ldots,y_M] \label{data_m}
\end{eqnarray}
be $M$ data points obtained from simulation of a dynamical system $x\mapsto T(x)$ or from an experiment, where $y_i=T(x_i)$. Let
\begin{eqnarray}\label{data_m_lifted}
\begin{aligned}
& {\G}_M={\bf \Psi}(X_M)^\top {\bf \Psi}(X_M)\\
& {\A}_M={\bf \Psi}(X_M)^\top {\bf \Psi}(Y_M)
\end{aligned}
\end{eqnarray}
be the data points in the lifted space $(\mathbb{R}^K)$, where the points $x_i$ and $y_i$ are mapped by the Koopman dictionary functions ${\bf \Psi}$.
Let
\begin{eqnarray}\label{Koopman_m_step}
{\bf K}_M = ({{\G}_{M}}+\lambda I)^{-1} {{\A}_{M}}
\end{eqnarray}
be the Koopman operator obtained by using the $M$ data points. Now, a new data point $(x_{M+1},y_{M+1})$ is acquired. The problem is to update the Koopman operator ${\K}_M$ to ${\bf K}_{M+1}$, without explicitly computing the inverse of $({\G}_{M+1}+\lambda I)$.

Note that (\ref{Koopman_m_step}) can be rewritten as
\begin{eqnarray}
{\A}_M = ({\G}_{M}+\lambda I){\K}_M  = \hat{\G}_M{\K}_M.
\end{eqnarray}
where 

\begin{eqnarray}
\begin{aligned}
\hat{\G}_M &= {\bf \Psi}(X_M)^\top {\bf \Psi}(X_M)+\lambda I\\
&=\sum_{i=1}^M{\bf \Psi}(X_i)^\top {\bf \Psi}(X_i) + \lambda I\\
{\A}_M &= {\bf \Psi}(X_M)^\top {\bf \Psi}(Y_M)=\sum_{i=1}^M{\bf \Psi}(X_i)^\top {\bf \Psi}(Y_i)
\end{aligned}
\end{eqnarray}
and ${\bf \Psi}(X_i)$ and ${\bf \Psi}(Y_i)$ are the $i$th columns of ${\bf \Psi}(X_M)$ and ${\bf \Psi}(Y_M)$ respectively.

Now, as the new data point $y_{M+1}=T(x_{M+1})$ streams in, the updated Koopman operator ${\K}_{M+1}$ satisfies
\begin{eqnarray}\label{Koopman_updated}
\hat{\G}_{M+1}{\K}_{M+1} = {\A}_{M+1},
\end{eqnarray}
where
\begin{eqnarray}
\begin{aligned}
& \hat{\G}_{M+1} = \sum_{i=1}^{M+1}{\bf \Psi}(X_i)^\top {\bf \Psi}(X_i) + \lambda I\\
& {\A}_{M+1}=\sum_{i=1}^{M+1}{\bf \Psi}(X_i)^\top {\bf \Psi}(Y_i).
\end{aligned}
\end{eqnarray}

Now,
\begin{align*}
\hat{\G}_{M+1} = & \sum_{i=1}^{M+1}{\bf \Psi}(X_i)^\top {\bf \Psi}(X_i) + \lambda I\\
=& \left(\sum_{i=1}^{M}{\bf \Psi}(X_i)^\top {\bf \Psi}(X_i)  + \lambda I\right) + {\bf \Psi}(X_{M+1})^\top {\bf \Psi}(X_{M+1})\\
=& \hat{\G}_{M} + {\bf \Psi}(X_{M+1})^\top {\bf \Psi}(X_{M+1}).
\end{align*}
Hence, using the Matrix Inversion Lemma, we have
\begin{eqnarray}\label{phi_iterate}
\hat{\G}_{M+1}^{-1} = \hat{\G}_M^{-1} - \frac{\hat{\G}_M^{-1}{\bf \Psi}(X_{M+1})^\top {\bf \Psi}(X_{M+1})\hat{\G}_M^{-1}}{1 + {\bf \Psi}(X_{M+1}) \hat{\G}_M^{-1} {\bf \Psi}(X_{M+1})^\top}.
\end{eqnarray}
Moreover,
\begin{eqnarray}\label{zm_iterate}\nonumber
{\A}_{M+1} &=& \sum_{i=1}^{M+1}{\bf \Psi}(X_i)^\top {\bf \Psi}(Y_i)\\
&=& {\A}_M + {\bf \Psi}(X_{M+1})^\top {\bf \Psi}(Y_{M+1}).
\end{eqnarray}
Hence, from (\ref{Koopman_updated}),

\begin{eqnarray}\label{Koopman_new}\nonumber
{\K}_{M+1} &=& \hat{\G}_{M+1}^{-1}{\A}_{M+1}\\ 
&=& \left(\hat{\G}_M^{-1} - \frac{\hat{\G}_M^{-1}{\bf \Psi}(X_{M+1})^\top {\bf \Psi}(X_{M+1})\hat{\G}_M^{-1}}{1 + {\bf \Psi}(X_{M+1}) \hat{\G}_M^{-1} {\bf \Psi}(X_{M+1})^\top}\right)\\
&\times & \left({\A}_M + {\bf \Psi}(X_{M+1})^\top {\bf \Psi}(Y_{M+1})\right).
\end{eqnarray}



Equation (\ref{Koopman_new}) gives the formula for updating the Koopman operator as new data streams in, without explicitly computing the inverse at every step, thus reducing the computational cost and hence improving efficiency.

\subsection{Initialization of the Algorithm}

Equation (\ref{Koopman_new}) gives the updated Koopman ${\K}_{M+1}$ operator in terms of quantities computed from the previous time step.
Note that the updated Koopman operator ${\K}_{M+1}$ depends on an inverse, namely, $\hat{\G}_M^{-1}$.
Hence, for computing the Koopman operator ${\K}_1$, one needs to initialize both $\hat{\G}_0$ and ${\A}_0$.
One potential way out of this situation is to compute the Koopman operator ${\K}_q$ using the initial $q$ data points $(x_i,y_i)$, $i=1,2,\cdots ,q$, $q<M$ as
\[{\K}_q = \hat{\G}_q^{-1}{\A}_q\]
and use the corresponding $\hat{\G}_q$ and ${\A}_q$ to compute the updated Koopman operators ${\K}_n$, $n>q$.

For practical applications, we set 
\[\hat{\G}_0 = \lambda I_K, \quad {\A}_0 = 0_K,\]
where $\lambda >0$, $I_K$ is the $K\times K$ identity matrix, and $0_K$ is the $K\times K$ zero matrix.

\begin{remark}
Choosing the initialization parameter $\delta$ can be tricky and usually one should run the algorithm multiple times, with different $\delta$ on a given training dataset and choose the one which has the lowest error for some validation dataset. 
\end{remark}

\begin{algorithm}[htp!]
\caption{Algorithm for online Koopman Operator computation using streaming data.}
\begin{enumerate}
\item{Fix the dictionary functions ${\bf \Psi} = [\psi_1, \cdots , \psi_K]$.}
\item{Initialize $\hat{\G}_0 = \lambda I_K$ and $ {\A}_0 = 0_K$.}
\item{As a new data point $(x_{M+1},y_{M+1})$ streams in, lift the data point to $\mathbb{R}^K$ using the dictionary function $\bf \Psi$.}
\item{Update ${\A}_{M}$ and $\hat{\G}_{M}^{-1}$ as
{\small
\begin{eqnarray*}
{\A}_{M+1} &=& {\A}_M + {\bf \Psi}(X_{M+1})^\top {\bf \Psi}(Y_{M+1})\\
\hat{\G}_{M+1}^{-1} &=& \hat{\G}_M^{-1} - \frac{\hat{\G}_M^{-1}{\bf \Psi}(X_{M+1})^\top {\bf \Psi}(X_{M+1})\hat{\G}_M^{-1}}{1 + {\bf \Psi}(X_{M+1}) \hat{\G}_M^{-1} {\bf \Psi}(X_{M+1})^\top}.
\end{eqnarray*}
}}
\item{Update the Koopman operator ${\K}_{M}$ to ${\K}_{M+1}$ as 
\begin{eqnarray*}
{\K}_{M+1} = \hat{\G}_{M+1}^{-1}{\A}_{M+1}.
\end{eqnarray*}
}
\end{enumerate}
\label{algo}
\end{algorithm}






\section{Simulation Results}\label{sec_simulation}
In this section, we demonstrate and discuss the properties of the proposed framework for iterative learning of dynamical models from noisy time-series data.
%
\subsection{Van der Pol Oscillator} 
Consider a noisy Van der Pol oscillator, whose equation of motion is given by 
\[\ddot{x}_t=\mu (1-x_t^2) - x_t + \sigma \xi_t,\]
where $x\in\mathbb{R}$ is the position variable, $\mu\geq 0$ is the damping co-efficient, $\sigma>0$ is a constant, and $\xi_t\in\mathbb{R}^2$ is independent and identically distributed Gaussian noise of zero mean and unit variance.
For simulation purposes, we chose $\mu = 0.8$ and $\sigma = 0.2$ and data were sampled at an interval of $0.01$ seconds. The phase portrait of the noisy Van der Pol oscillator is shown in Fig. \ref{fig_noisy_phase_van_der_pol}. 
\begin{figure}[htp!]
\centering
\includegraphics[scale=.4]{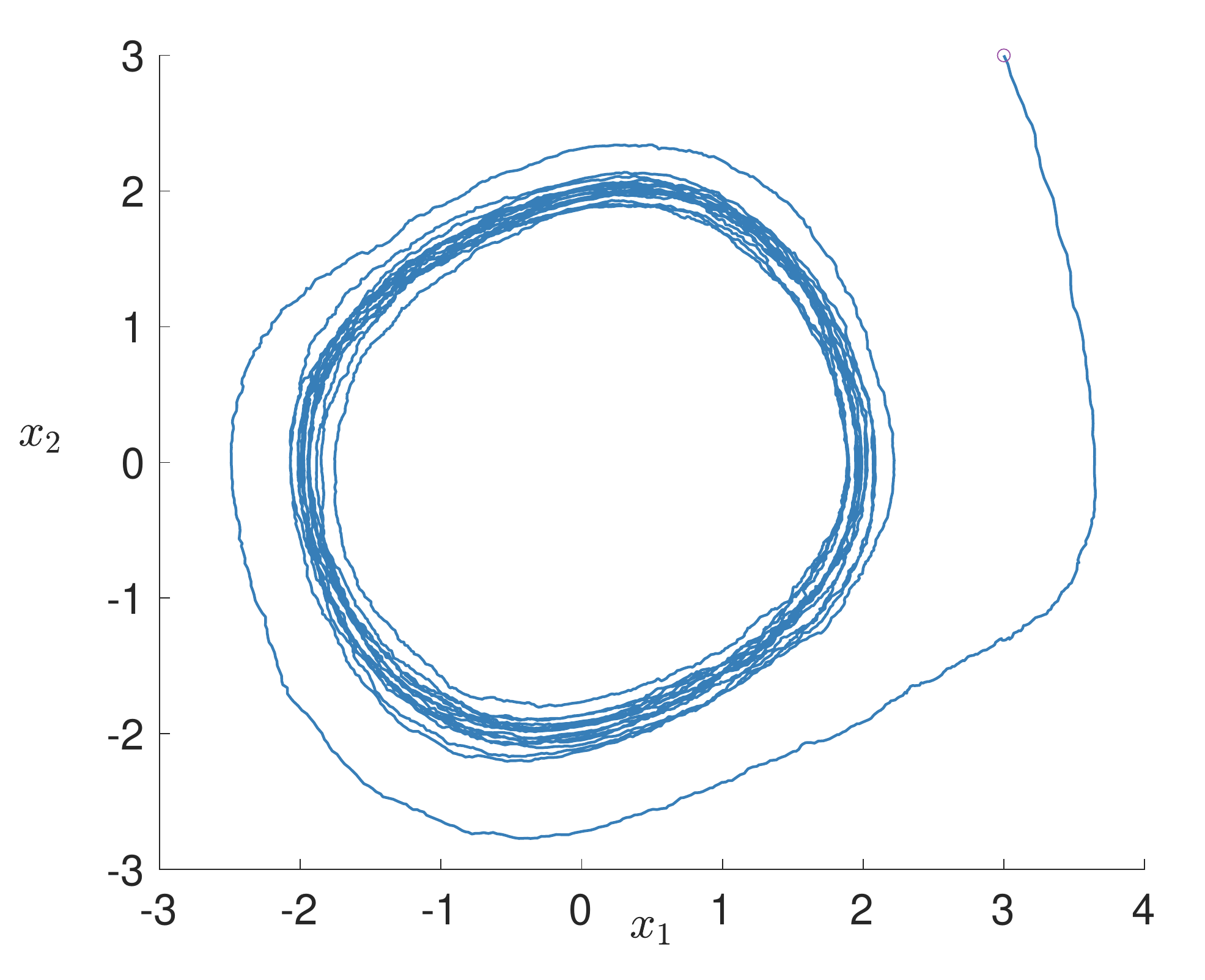}
\caption{Phase portrait of noisy Van der Pol oscillator.}\label{fig_noisy_phase_van_der_pol}
\end{figure}

We considered a single trajectory data, from a random initial condition, as the training data and used 40 radial Gaussian basis functions as the Koopman observables.
Fig.~\ref{fig_van_der_pol_attractor} shows that the RR-EDMD (Recursive Robust Extended Dynamic Mode Decomposition)  algorithm recovers the stable limit cycle of the Van der Pol oscillator as new data-points stream in and the Koopman operator gets updated.
In particular, Fig.~\ref{fig_van_der_pol_attractor}(a) shows that the eigenvector corresponding to the unit eigenvalue of the Koopman operator obtained after 500 data points have streamed in has partially identified the stable limit cycle.
However, as more data points stream in, the identified region of the limit cycle grows and in Fig.~\ref{fig_van_der_pol_attractor}(b) we find that with 2000 data points, the complete limit cycle has been identified.
\begin{figure}[htp!]
\centering
\subfigure[]{\includegraphics[scale=.4]{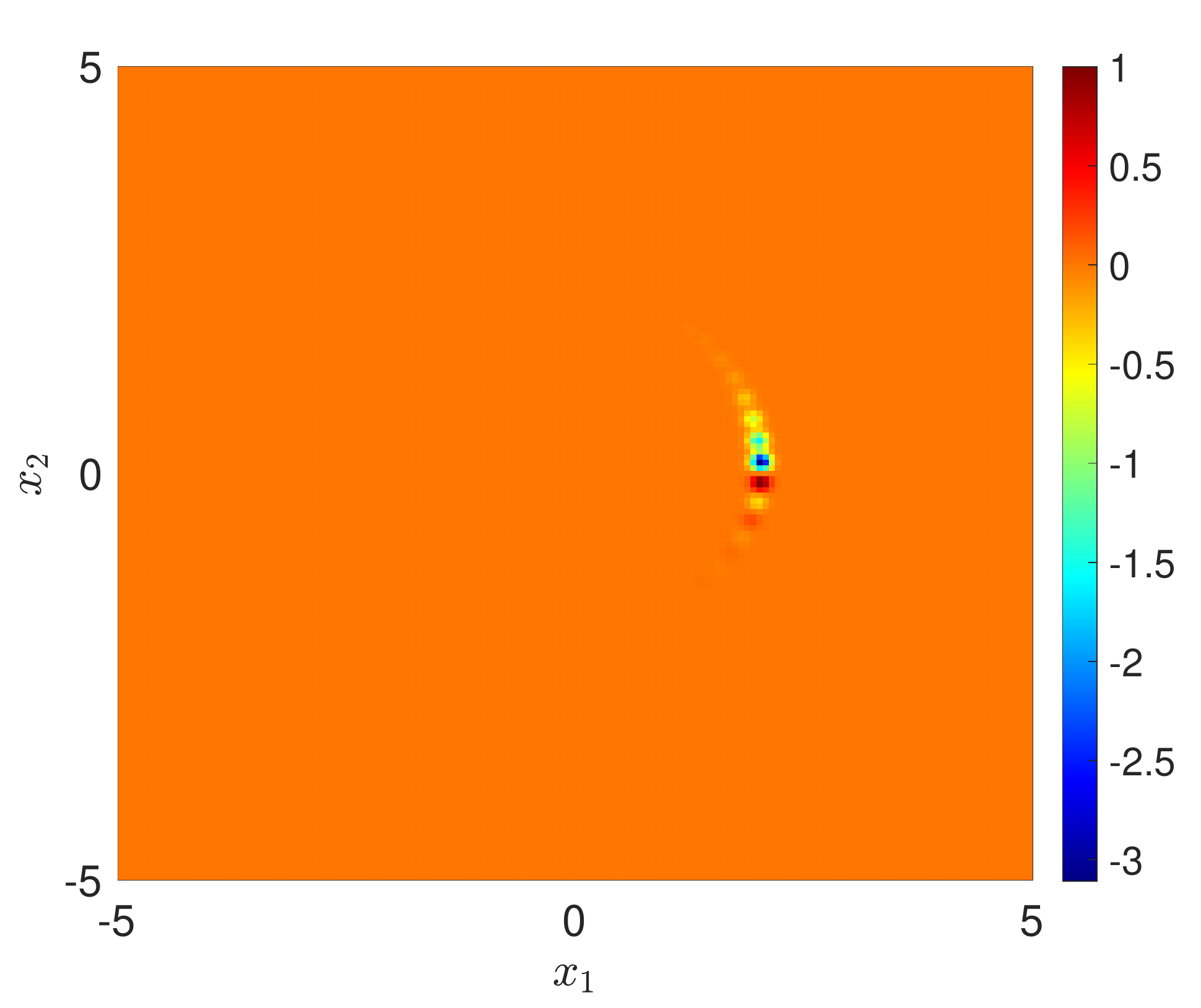}}
\subfigure[]{\includegraphics[scale=.4]{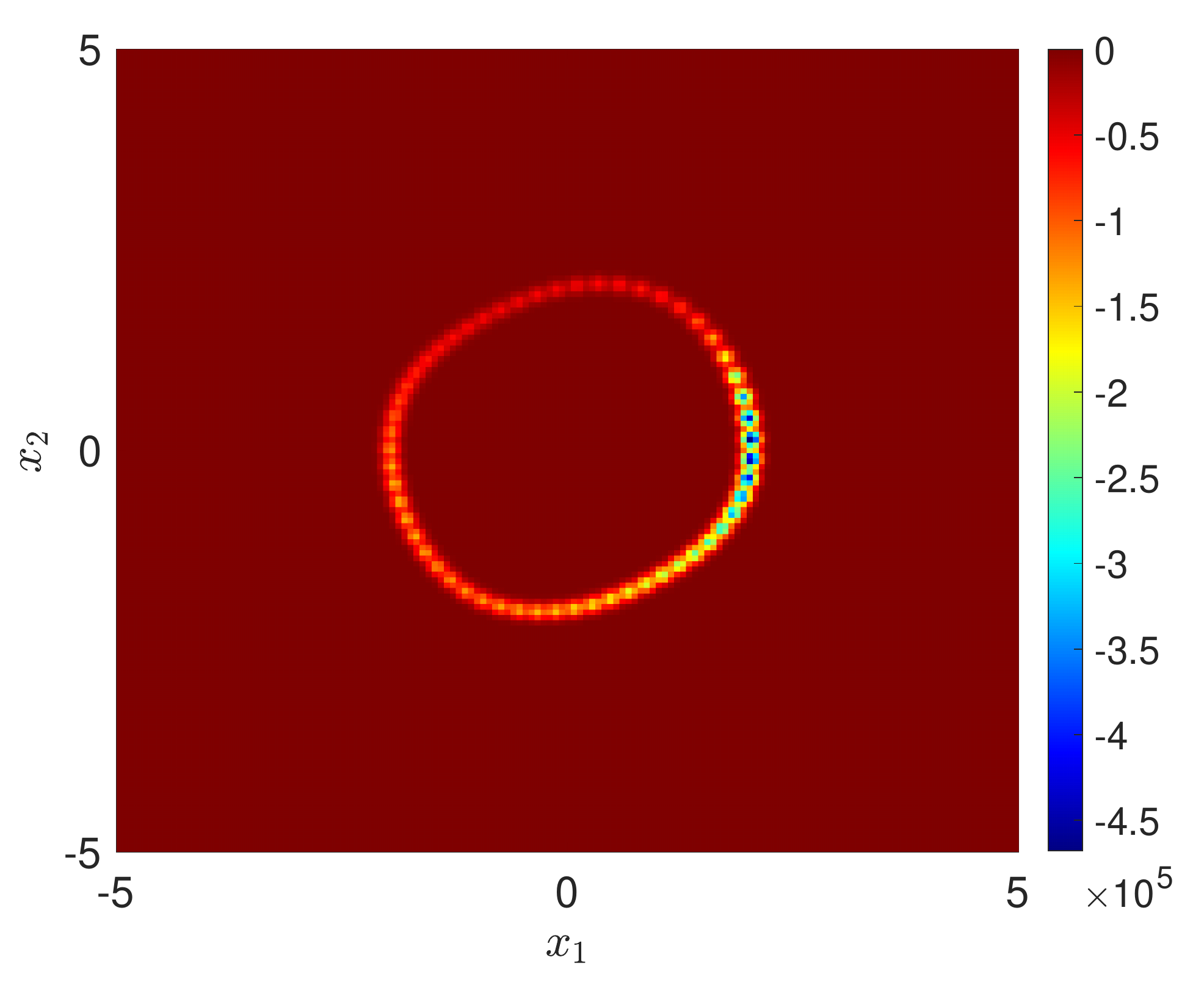}}
\caption{(a) The Koopman spectrum, computed from 500-time points data identifies the stable limit cycle of the Van der Pol oscillator partially. (b) After 2000 iterations, the Koopman spectra recover the entire stable limit cycle.}\label{fig_van_der_pol_attractor}
\end{figure}

However, the biggest advantage of the RR-EDMD formulation is the fact that it is much faster compared to EDMD.
To compare the computational cost of RR-EDMD and EDMD, we present in Fig.~\ref{fig_EDMD_REDMD_comp_time_comparison_van_der_pol} the computation times of both the algorithms for learning the Koopman operator as new data-points stream in.
It can be seen the computation time of the RR-EDMD algorithm varies linearly with the number of iterates, whereas for EDMD it varies almost quadratically.
This establishes a clear advantage of the RR-EDMD algorithm over EDMD and it is this fact that facilitates the use of RR-EDMD framework for online real-time learning of dynamical systems from streaming data.
\begin{figure}[htp!]
\centering
\includegraphics[scale=.4]{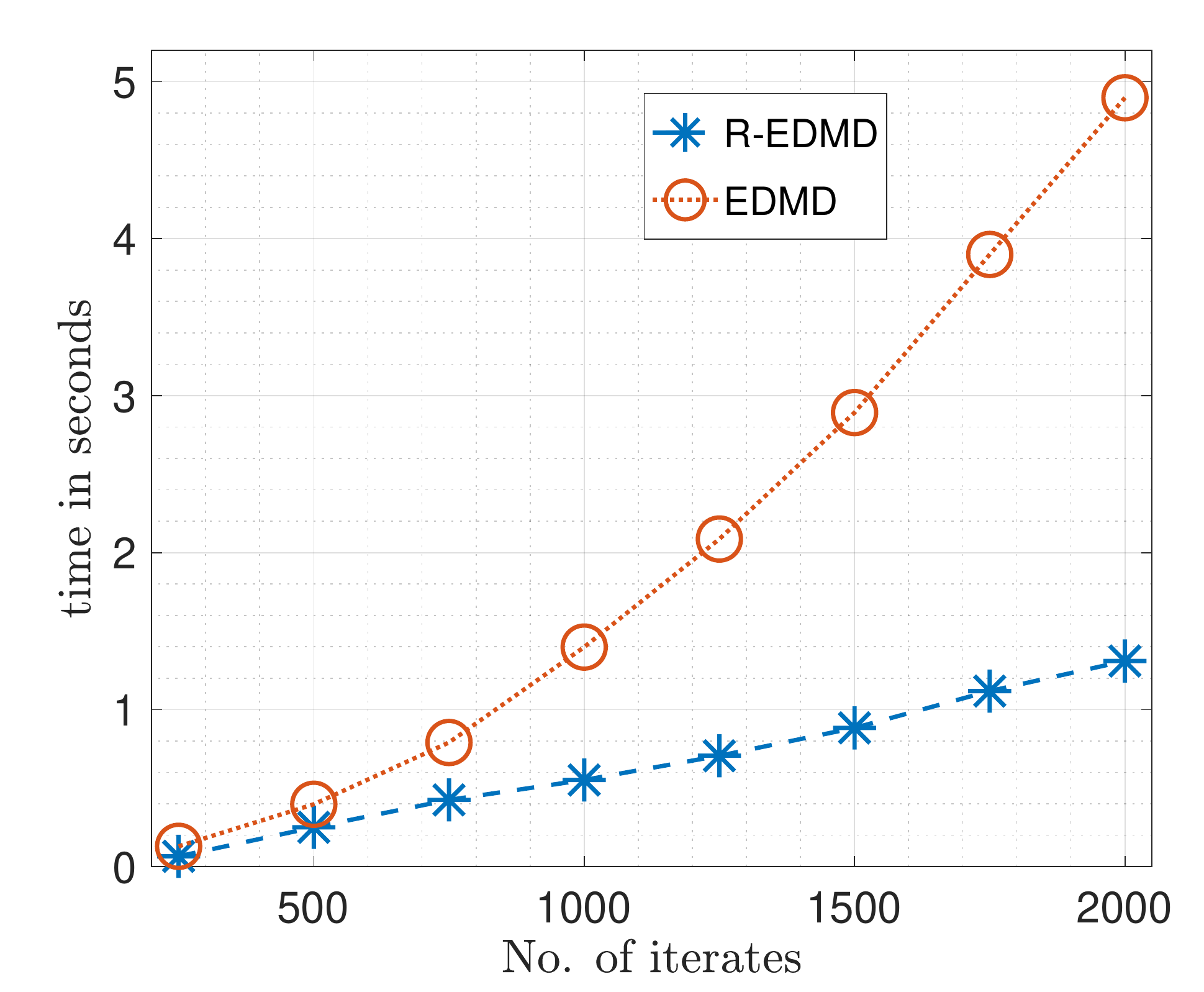}
\caption{Comparison of computation time for recursive learning using normal EDMD and RR-EDMD.}\label{fig_EDMD_REDMD_comp_time_comparison_van_der_pol}
\end{figure}
\subsection{IEEE 68 Bus System}
We now discuss the application of the proposed RR-EDMD algorithm to a power system.
We consider synthetic time-series measurements corresponding to the 68-bus power network with 16 generators shown in Fig.~\ref{fig:68bus}.
The synthetic measurements are generated using GridSTAGE \cite{nandanoori2020model}, a power system toolbox-based simulation platform that emulates PMU measurements corresponding to load changes or faults or adversarial actions.
\begin{figure}[h!]
    \centering
    \includegraphics[scale=.7]{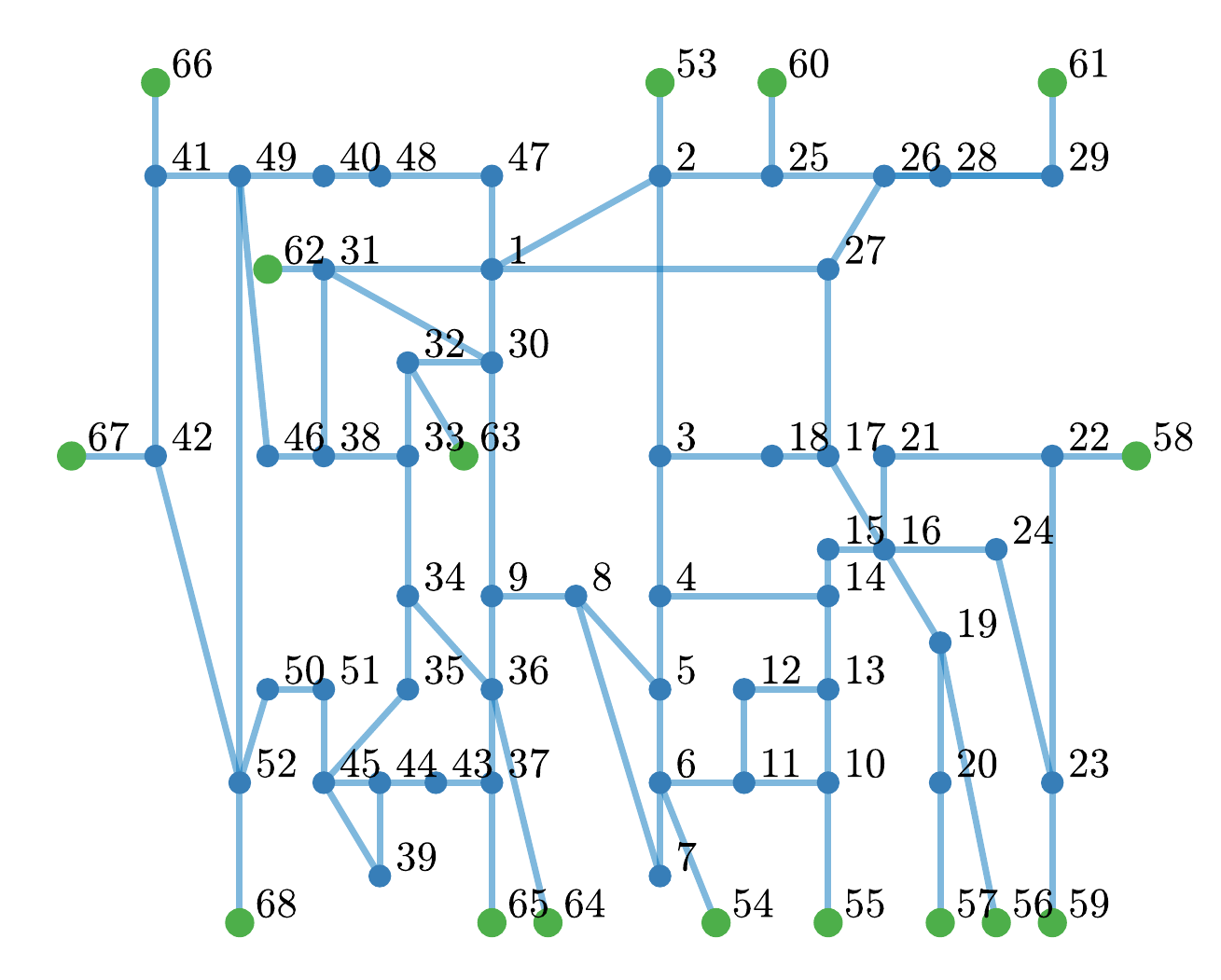}
    \caption{One line diagram representation of IEEE 68 bus network.}
    \label{fig:68bus}
\end{figure}

For the real-time learning of power system dynamics, random and large load changes are created using GridSTAGE.
Although the synthetic time-series measurements obtained through GridSTAGE are free of noise, the practical PMU measurements will not be due to the effect of communication uncertainties.
Hence, random noise is artificially added to the synthetic PMU measurements to mimic practical use cases and ensure the proper applicability of the proposed real-time learning framework.
The noise-corrupted PMU measurements are shown in Fig.~\ref{fig_power_data}. Since the operating point of the underlying power network keeps changing due to random load changes or faults or due to the change in the operating conditions, the dynamics learnt from PMU data will not be valid for all operating conditions and need to be updated.
Therefore, the proposed recursive learning method is particularly useful here for continuously learning the system dynamics from PMU data. 
\begin{figure}
\centering
\includegraphics[scale=.45]{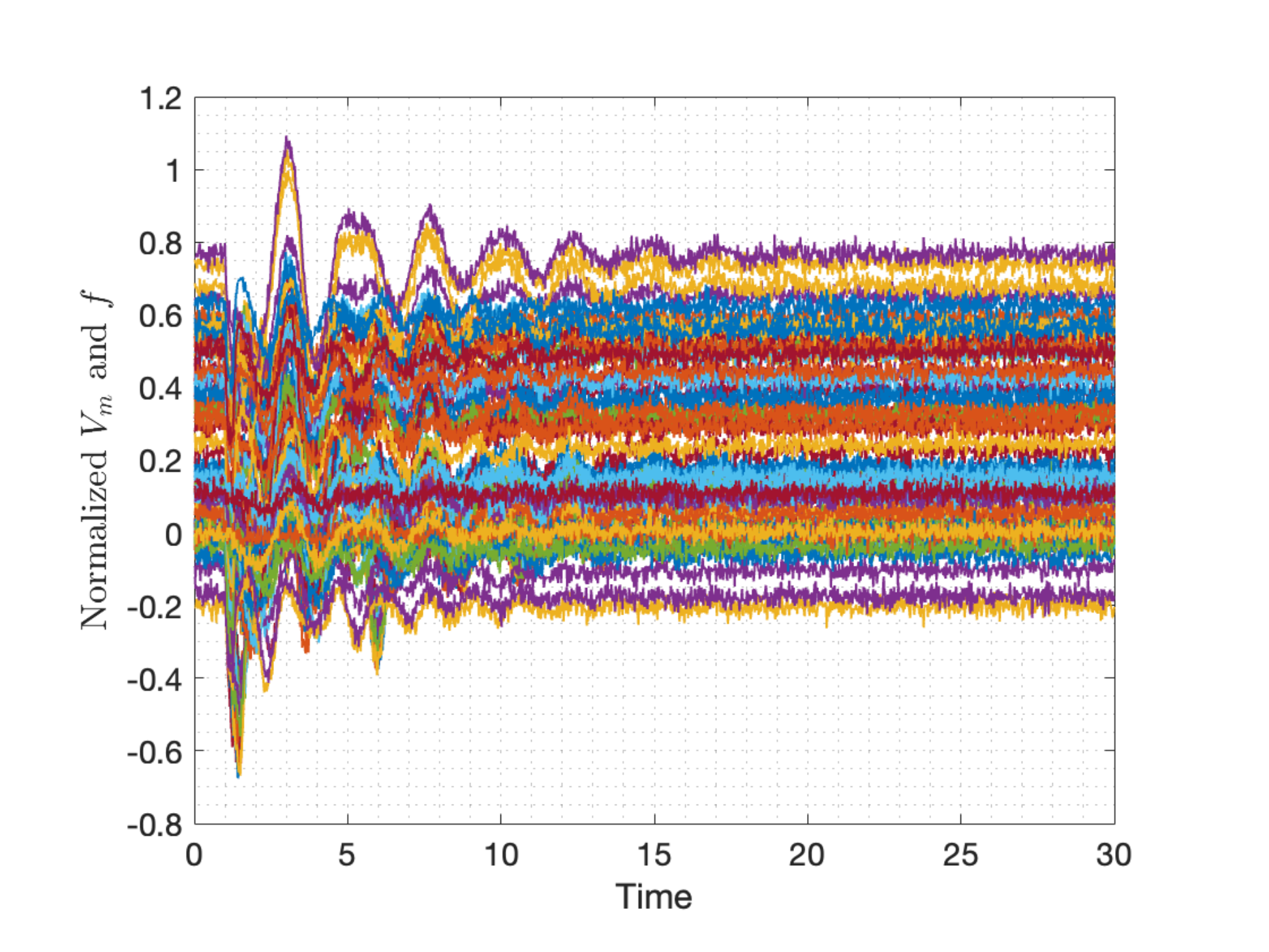}
\caption{Noisy time-series data of frequency and voltage obtained from PMUs.}\label{fig_power_data}
\end{figure}

The PMU measurements record data at a rate of 40-80 measurements per second. In this study, we fix its rate to 50 measurements per second and each PMU measures several system states such as frequency, rate of change of frequency, voltage magnitude, voltage angle, current flow through the connected branches are available. Furthermore, to capture the noise in the real-time PMU measurements, we manually add the noise to the measurements such that the signal-to-noise ration (SNR) of the resultant data is 85dB. To demonstrate the proposed robust recursive Koopman operator to understand the system evolution from noisy data in real-time, we perform the analysis considering the important states, frequency and voltage magnitude. 

Considering the initialization given in algorithm~\ref{algo}, a robust Koopman operator corresponding to the IEEE 68 bus system representing the underlying load changes is computed in real-time from streaming noisy PMU data. As the frequency and voltage magnitudes are considered at each bus, there is a total of 136 states and we use 150 Gaussian radial functions as observables, to compute the robust Koopman using RR-EDMD.
Fig.~\ref{fig_eig_68_bus} (a), (b) and (c) respectively show the eigenvalues of the learnt Robust Koopman operator after 100, 500 and 1000 measurements. Moreover, we can observe that the robust Koopman updated with every new measurement yields a stable representation (all eigenvalues lying inside the unit circle).
Essentially, applying the proposed RR-EDMD, the power system dynamics is learnt using measurements from a few seconds.
However, the Koopman operator learning is a continuous process as the PMU data streams in and the power system properties such as stability margins can be studied at any given time-point.  
Furthermore, we compared the computational time of RR-EDMD against that of EDMD, shown in Fig.~\ref{fig_EDMD_REDMD_comp_time_comparison_68_bus}.
It can be seen that the proposed RR-EDMD outperforms the standard EDMD method.

\begin{figure}[htp!]
\centering
\subfigure[]{\includegraphics[scale=.4]{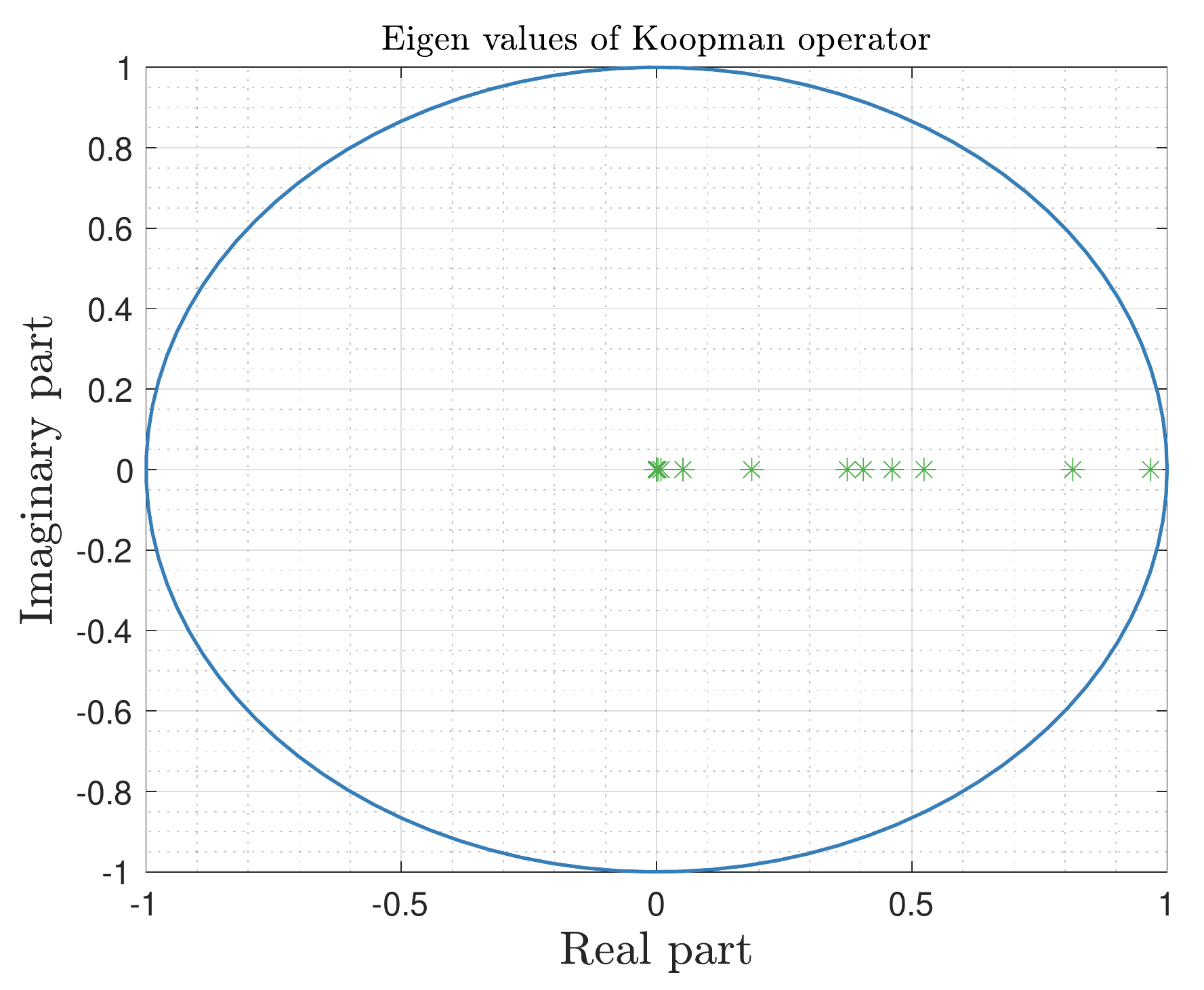}}
\subfigure[]{\includegraphics[scale=.4]{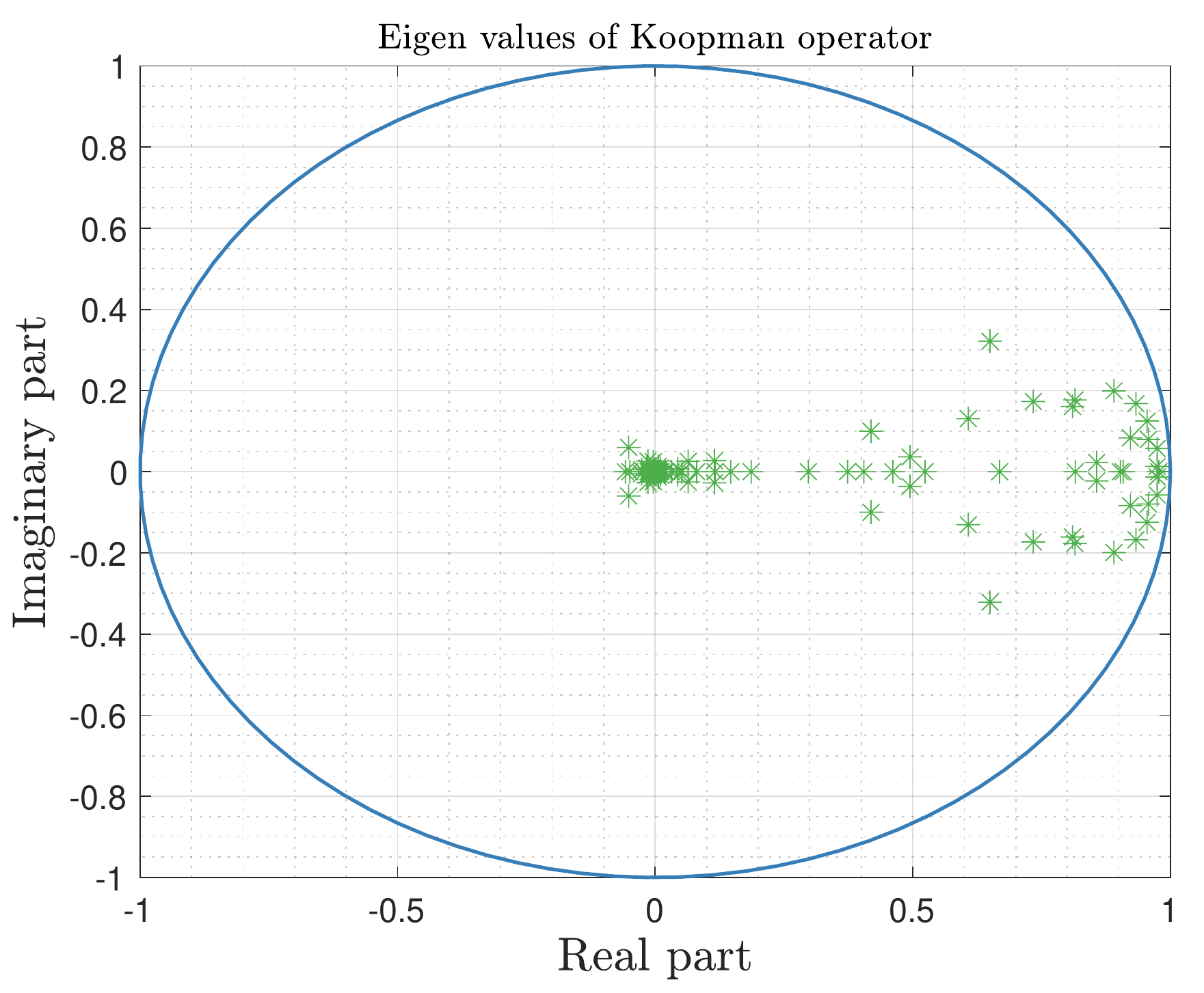}}
\subfigure[]{\includegraphics[scale=.4]{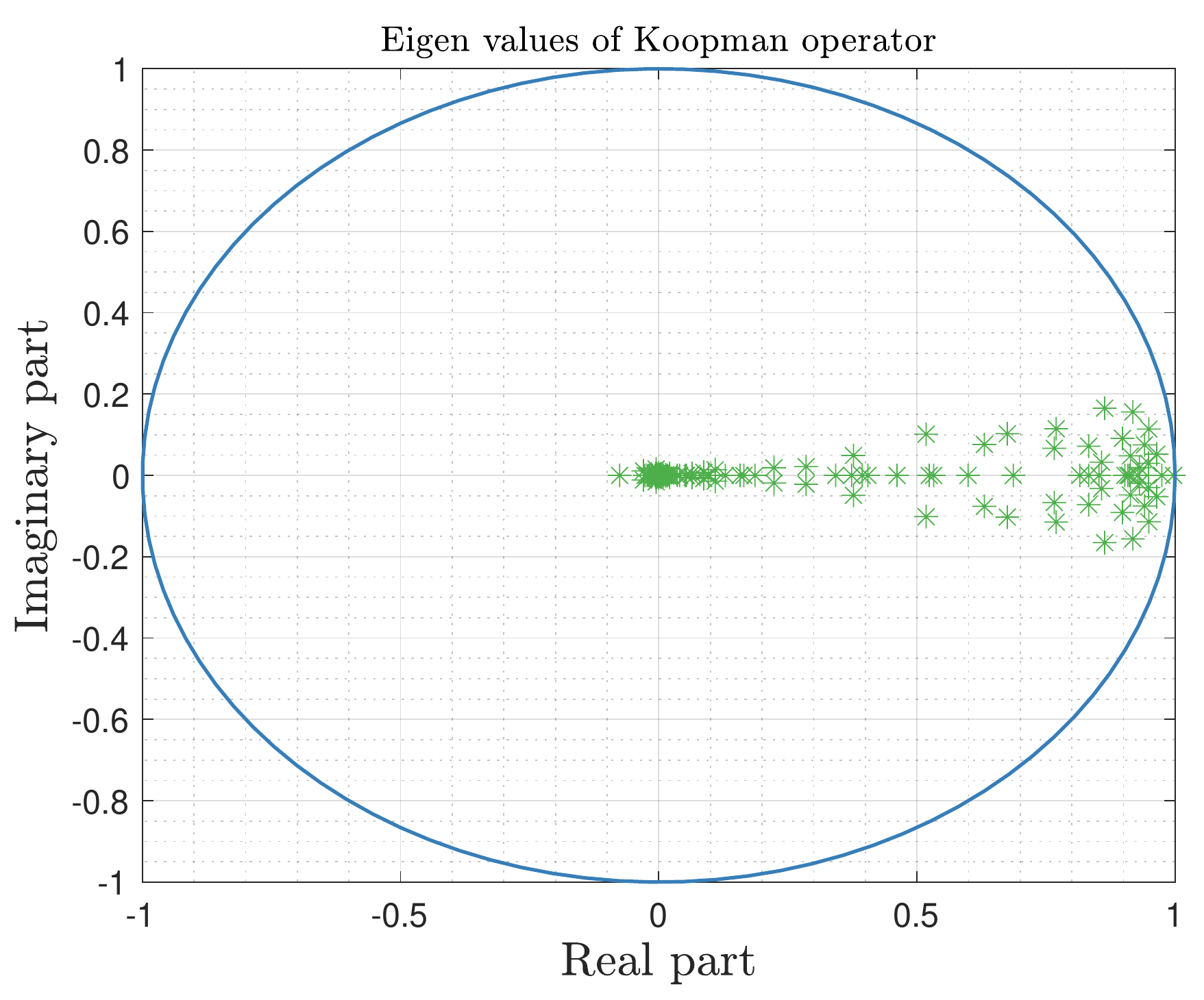}}
\caption{Eigenvalues of the Koopman operator after (a) 100 iterates; (b) 500 iterates; (c) 1000 iterates.}\label{fig_eig_68_bus}
\end{figure}

\begin{figure}[htp!]
\centering
\includegraphics[scale=.4]{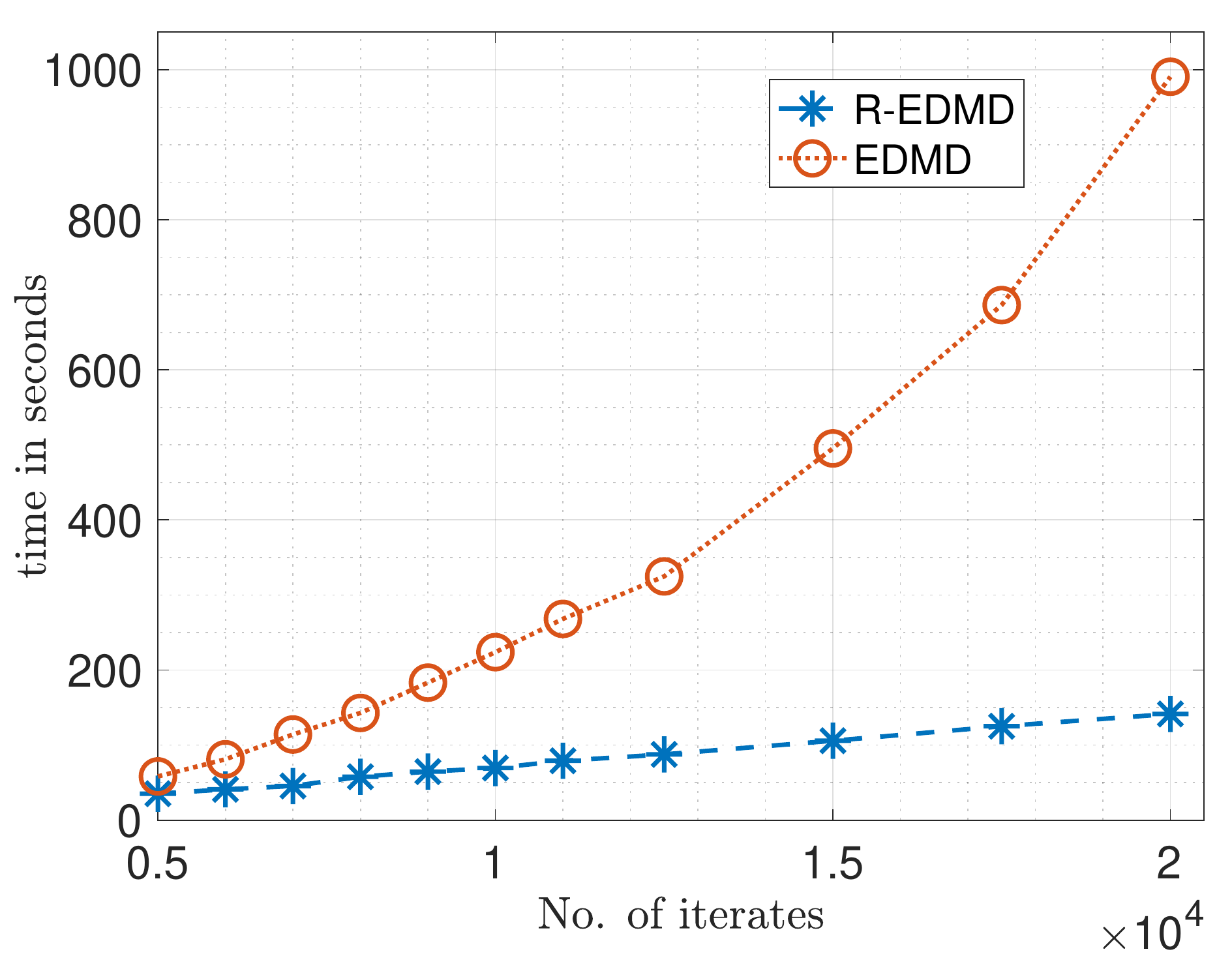}
\caption{Comparison of computation time for recursive learning using normal EDMD and RR-EDMD.}\label{fig_EDMD_REDMD_comp_time_comparison_68_bus}
\end{figure}

\subsection{Scalability of RR-EDMD with Dimension of the Underlying System}

In the previous two examples, we compared the computation times of the proposed RR-EDMD method and the existing EDMD method and found that RR-EDMD outperforms EDMD.
In this subsection, we analyze how the scalability of the RR-EDMD algorithm as the number of states in the underlying system increases.
In particular, we considered a ring network (Fig.~\ref{fig_scalability_van_der_pol}(a)) of Van der Pol oscillators and recorded the computation times of the Koopman operator using the RR-EDMD algorithm as the number of oscillators in the system increases.
The equation of motion of the $i$th oscillator is given by
\[\ddot{x}_i(t)=\mu (1-x_i^2(t)) - x_i(t) + \mathcal{L}_ix(t) + \sigma \xi_i(t),\]
where $x_i(t)$ is the position variable of the $i$th oscillator, $x(t)=[x_1(t), \cdots , x_n(t)]^\top$, $\mathcal{L}_i$ is the $i$th row of the network Laplacian, $\mu\geq 0$ is the damping constant, $\sigma>0$ is a constant and $\xi_i(t)\in\mathbb{R}^2$ is an i.i.d. Gaussian noise of zero mean and unit variance.
We varied the number of oscillators from 50 to 500, so that the number of states of the system varied from 100 to 1000.
However, it is to be noted that the Koopman operator is computed in the space of the Koopman observables.
In this example, we used 15 radial Gaussian basis functions per oscillator, so that the size of the Koopman operator varied from $750 \times 750$ to $7500 \times 7500$.
For simulation purposes, we used the same $\mu=0.8$ and $\sigma=0.2$ for all the oscillators.

\begin{figure}[htp!]
\centering
\subfigure[]{\includegraphics[scale=.65]{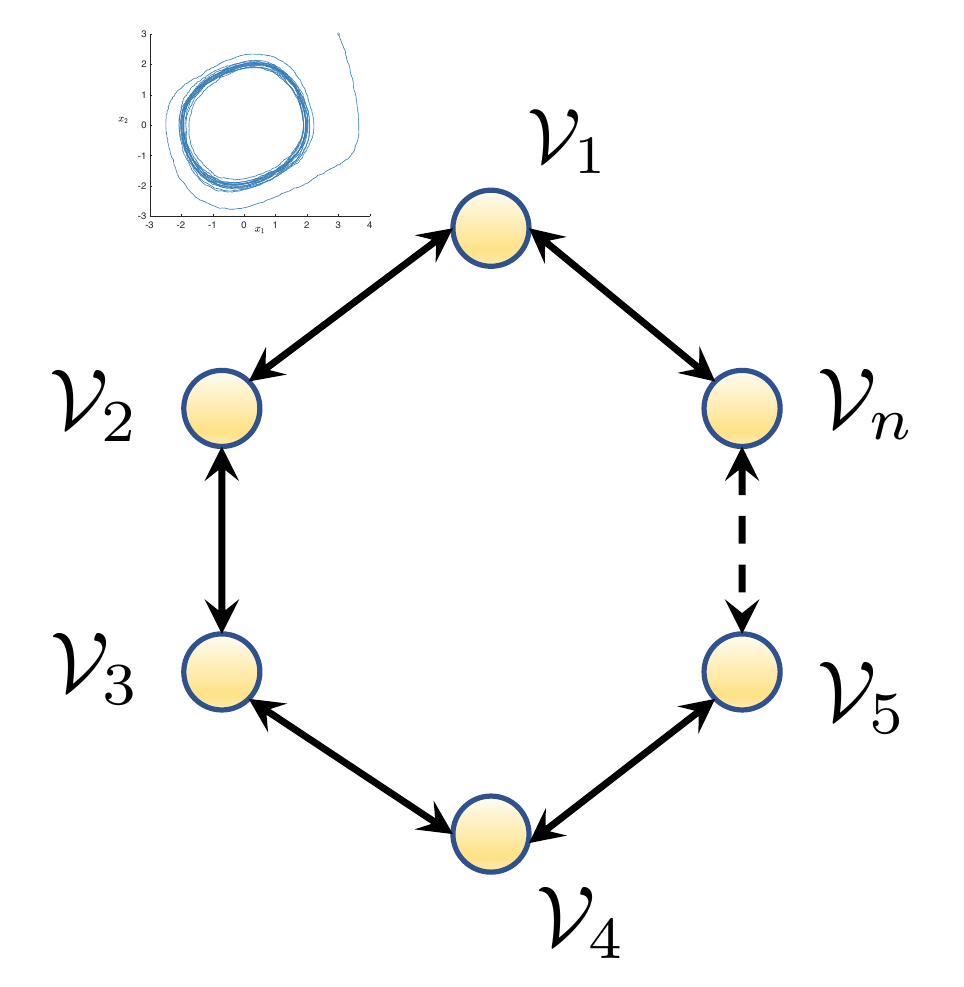}}
\subfigure[]{\includegraphics[scale=.35]{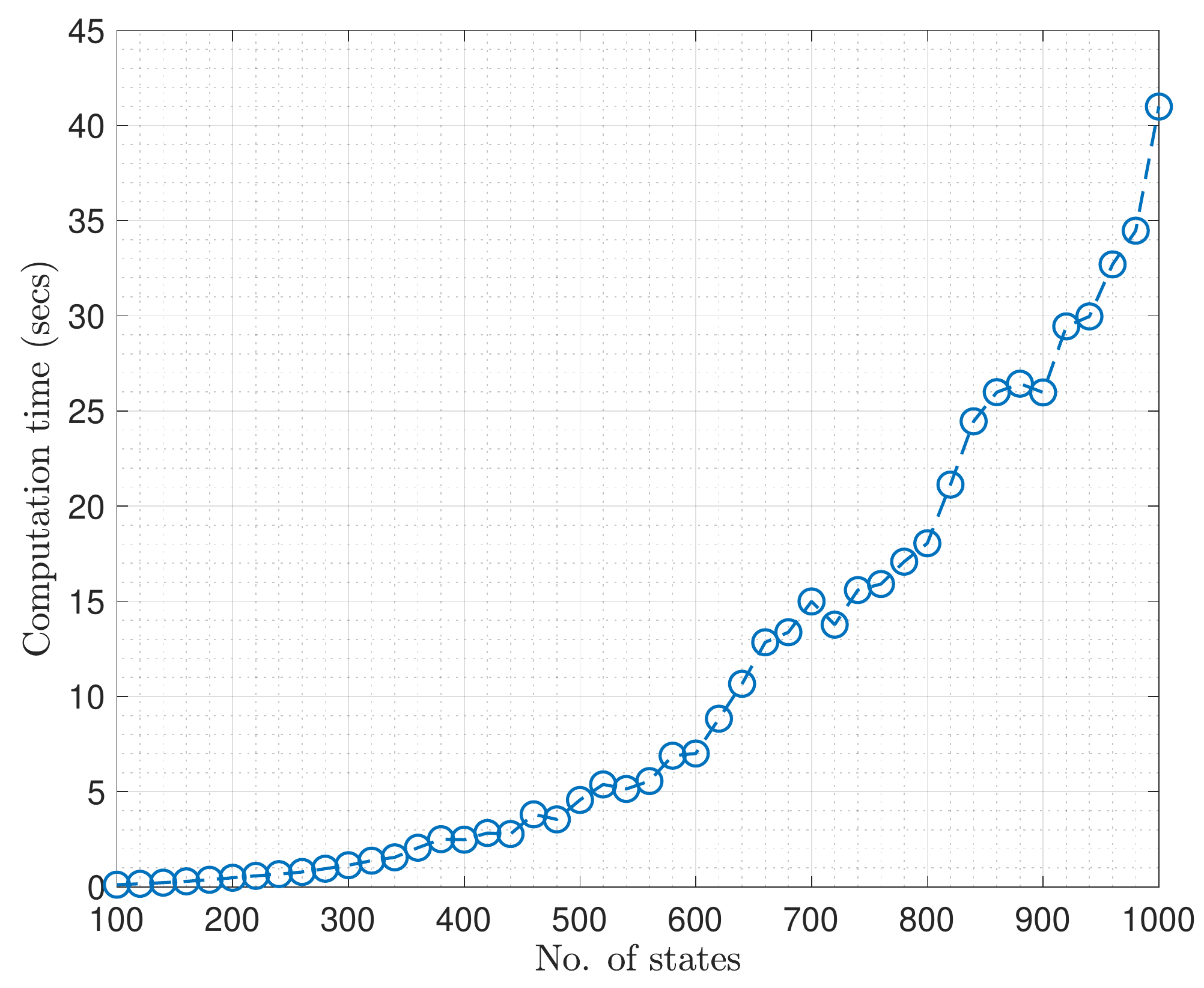}}
\caption{(a) Network of Van der Pol oscillators. (b) Computation time for recursive learning using RR-EDMD, as the number of states in the system increases.}\label{fig_scalability_van_der_pol}
\end{figure}

Fig.~\ref{fig_scalability_van_der_pol}(b) shows how the computation time varies with the number of states.
On the face of it, it seems that the computation time varies almost quadratically with the number of states.
However, it should be noted as the number of states increases, so does the number of dictionary functions, leading to the quadratic nature of the plot. 
However, if the number of dictionary functions remains constant, the computation time varies linearly with the number of states (Fig. \ref{fig_EDMD_REDMD_comp_time_comparison_van_der_pol}).
Since the RR-EDMD algorithm scales almost linearly with the dimension of the system, it can be used for the identification of large-dimensional systems in real time as well.



\section{Conclusions}\label{sec_conclusions}
In this paper, we address the important problem of learning the dynamics of a general dynamical system, from noisy measurements in real time.
In particular, we use the Koopman operator framework for data-driven learning, so that one obtains a linear system representation of the underlying dynamical system.
We resort to Robust Koopman operator estimation to mitigate the effect of measurement noise and we propose an iterative algorithm for recursively learning the Robust Koopman operator from streaming data.
We show that the proposed framework is substantially faster than the existing EDMD algorithm, thus making it practical to learn the dynamics in an online fashion and in real time.
We also demonstrate the efficiency of the proposed algorithm by applying it to identify the Van der Pol oscillator, the IEEE 68 bus system and a ring network of Van der Pol oscillators.

\section*{Acknowledgements}
This work was supported by the U.S. Department of Energy (DOE) Office of Science, Office of Advanced Scientific Computing Research (ASCR) as part of the Multifaceted Mathematics for Rare, Extreme Events in Complex Energy and Environment Systems (MACSER) project.
Pacific Northwest National Laboratory is operated by Battelle for the DOE under Contract DE-AC05-76RL01830.

\bibliography{subhrajit_online_identification}

\begin{thebibliography}{10}
\expandafter\ifx\csname url\endcsname\relax
  \def\url#1{\texttt{#1}}\fi
\expandafter\ifx\csname urlprefix\endcsname\relax\def\urlprefix{URL }\fi
\expandafter\ifx\csname href\endcsname\relax
  \def\href#1#2{#2} \def\path#1{#1}\fi

\bibitem{principia_newton}
I.~Newton, Principia Mathematica, 1687.

\bibitem{Lasota}
A.~Lasota, M.~C. Mackey, Chaos, Fractals, and Noise: Stochastic Aspects of
  Dynamics, Springer-Verlag, New York, 1994.

\bibitem{mezic_koopmanism}
M.~Budisic, R.~Mohr, I.~Mezic, Applied koopmanism, Chaos 22 (2012) 047510--32.

\bibitem{EDMD_williams}
M.~O. Williams, I.~G. Kevrekidis, C.~W. Rowley, A data--driven approximation of
  the koopman operator: Extending dynamic mode decomposition, Journal of
  Nonlinear Science 25~(6) (2015) 1307--1346.

\bibitem{mezic_koopman_stability}
A.~Mauroy, I.~Mezi{\'c}, A spectral operator-theoretic framework for global
  stability, in: Proc. of IEEE Conference of Decision and Control, Florence,
  Italy, 2013.

\bibitem{sinha_robust_dmd_journal}
S.~Sinha, B.~Huang, U.~Vaidya, {On robust computation of Koopman operator and
  prediction in random dynamical systems}, Journal of Nonlinear Science 30~(5)
  (2020) 2057--2090.

\bibitem{sinha_equivariant_ifac}
S.~Sinha, S.~P. Nandanoori, E.~Yeung, {Koopman operator methods for global
  phase space exploration of equivariant dynamical systems}, IFAC-PapersOnLine
  53~(2) (2020) 1150--1155.

\bibitem{nandanoori2020data}
S.~P. Nandanoori, S.~Sinha, E.~Yeung, Data-driven operator theoretic methods
  for global phase space learning, in: 2020 American Control Conference (ACC),
  IEEE, 2020, pp. 4551--4557.

\bibitem{nandanoori2021data}
S.~P. Nandanoori, S.~Sinha, E.~Yeung, Data-driven operator theoretic methods
  for phase space learning and analysis, arXiv e-prints (2021) arXiv--2106.

\bibitem{sinha_data_driven_control_arxiv}
S.~Sinha, S.~P. Nandanoori, J.~Drgona, D.~Vrabie, Data-driven stabilization of
  discrete-time control-affine nonlinear systems: A koopman operator approach,
  accepted for publication in ECC, arXiv preprint arXiv:2203.14114 (2022).

\bibitem{huang2018feedback}
B.~Huang, X.~Ma, U.~Vaidya, Feedback stabilization using koopman operator, in:
  2018 IEEE Conference on Decision and Control (CDC), IEEE, 2018, pp.
  6434--6439.

\bibitem{sinha_sparse_koopman_acc}
S.~Sinha, U.~Vaidya, E.~Yeung, On computation of koopman operator from sparse
  data, in: 2019 American Control Conference (ACC), IEEE, 2019, pp. 5519--5524.

\bibitem{sootla2017pulse}
A.~Sootla, D.~Ernst, Pulse-based control using koopman operator under
  parametric uncertainty, IEEE Transactions on Automatic Control 63~(3) (2017)
  791--796.

\bibitem{harrison2021stability}
J.~Harrison, E.~Yeung, Stability analysis of parameter varying genetic toggle
  switches using koopman operators, Mathematics 9~(23) (2021) 3133.

\bibitem{eisenhower2010decomposing}
B.~Eisenhower, T.~Maile, M.~Fischer, I.~Mezic, Decomposing building system data
  for model validation and analysis using the koopman operator, in: Proceedings
  of the National IBPSAUSA Conference, New York, USA, 2010.

\bibitem{korda_mezic_predictor}
M.~Korda, I.~Mezi{\'c}, Linear predictors for nonlinear dynamical systems:
  Koopman operator meets model predictive control, arXiv preprint
  arXiv:1611.03537 (2016).

\bibitem{slawinska2019quantum}
J.~Slawinska, A.~Ourmazd, D.~Giannakis, A quantum mechanical approach for data
  assimilation in climate dynamics, in: International Conference on Machine
  Learning Workshop on, 2019.

\bibitem{bruder2019modeling}
D.~Bruder, B.~Gillespie, C.~D. Remy, R.~Vasudevan, Modeling and control of soft
  robots using the koopman operator and model predictive control, arXiv
  preprint arXiv:1902.02827 (2019).

\bibitem{abraham2019active}
I.~Abraham, T.~D. Murphey, Active learning of dynamics for data-driven control
  using koopman operators, IEEE Transactions on Robotics 35~(5) (2019)
  1071--1083.

\bibitem{marrouch2020data}
N.~Marrouch, J.~Slawinska, D.~Giannakis, H.~L. Read, Data-driven koopman
  operator approach for computational neuroscience, Annals of Mathematics and
  Artificial Intelligence 88~(11) (2020) 1155--1173.

\bibitem{susuki2016applied}
Y.~Susuki, I.~Mezic, F.~Raak, T.~Hikihara, Applied koopman operator theory for
  power systems technology, Nonlinear Theory and Its Applications, IEICE 7~(4)
  (2016) 430--459.

\bibitem{sinha2019information}
S.~Sinha, P.~Sharma, U.~Vaidya, V.~Ajjarapu, On information transfer-based
  characterization of power system stability, IEEE Transactions on Power
  Systems 34~(5) (2019) 3804--3812.

\bibitem{nandanoori2022graph}
S.~P. Nandanoori, S.~Guan, S.~Kundu, S.~Pal, K.~Agarwal, Y.~Wu, S.~Choudhury,
  Graph neural network and koopman models for learning networked dynamics: A
  comparative study on power grid transients prediction, IEEE Access 10 (2022)
  32337--32349.

\bibitem{kropp2006system}
T.~Kropp, System threats and vulnerabilities [power system protection], IEEE
  Power and Energy Magazine 4~(2) (2006) 46--50.

\bibitem{sinha_online_learning_PES}
S.~Sinha, S.~P. Nandanoori, E.~Yeung, Data driven online learning of power
  system dynamics, in: 2020 IEEE Power \& Energy Society General Meeting
  (PESGM), IEEE, 2020, pp. 1--5.

\bibitem{robust_DMD_ACC}
S.~Sinha, B.~Huang, U.~Vaidya, {Robust Approximation of Koopman Operator and
  Prediction in Random Dynamical Systems}, in: 2018 Annual American Control
  Conference (ACC), IEEE, 2018, pp. 5491--5496.

\bibitem{nandanoori2020model}
S.~P. Nandanoori, S.~Kundu, S.~Pal, K.~Agarwal, S.~Choudhury, {Model-agnostic
  algorithm for real-time attack identification in power grid using Koopman
  modes}, in: 2020 IEEE International Conference on Communications, Control,
  and Computing Technologies for Smart Grids (SmartGridComm), IEEE, 2020, pp.
  1--6.

\end{thebibliography}

\end{document}